\documentclass{amsart}

\makeatletter
\def\subsection{\@startsection{subsection}{2}
  \z@{.5\linespacing\@plus.7\linespacing}{.5\linespacing}
  {\normalfont\bfseries}}
\makeatother

\makeatletter
\def\@defaultbiblabelstyle#1{[#1]}
\makeatother

\makeatletter
\def\@setauthors{
  \begingroup
  \def\thanks{\protect\thanks@warning}
  \trivlist
  \centering\footnotesize \@topsep30\p@\relax
  \advance\@topsep by -\baselineskip
  \item\relax
  \author@andify\authors
  \def\\{\protect\linebreak}
  \authors
  \ifx\@empty\contribs
  \else
    ,\penalty-3 \space \@setcontribs
    \@closetoccontribs
  \fi
  \endtrivlist
  \endgroup
}
\def\@settitle{\begin{center}
  \baselineskip14\p@\relax
    \bfseries
  \@title
  \end{center}
}
\makeatother

\usepackage{amssymb}
\usepackage{amsxtra}
\usepackage{mathtools}
\usepackage[colorlinks=true, citecolor=blue]{hyperref}
\usepackage{upref}
\usepackage{soul}
\usepackage{enumitem}
\setlist[enumerate]{font=\normalfont}
\usepackage{verbatim}

\newcommand{\dd}{\mathrm{d}}

\newtheorem{theorem}{Theorem}[section]
\newtheorem{lemma}[theorem]{Lemma}
\newtheorem{proposition}[theorem]{Proposition}
\newtheorem{corollary}[theorem]{Corollary}

\theoremstyle{definition}
\newtheorem{definition}[theorem]{Definition}

\theoremstyle{remark}
\newtheorem{remark}[theorem]{Remark}

\numberwithin{equation}{section}
\allowdisplaybreaks[4]

\usepackage{xcolor}
\definecolor{zsxpurple}{RGB}{128,0,128}

\definecolor{zswgreen}{RGB}{0,128,0}

\definecolor{chybrown}{RGB}{128,0,0}

\begin{document}

\title{Solomon--Terao Algebras: Top Degree, Dimension, and Homology}

\author{Hanyu Cai}
\address{School of Mathematical Sciences, Peking University, Yiheyuan Rd. 5, Haidian, Beijing 100871, China}
\email{2300010801@stu.pku.edu.cn}

\author{Haiyang Chen}
\address{School of Mathematical Sciences, Nankai University, 94 Weijin Road, Nankai District, Tianjin 300071, China}
\email{hai@mail.nankai.edu.cn}

\author{Shixuan Zeng}
\address{School of the Gifted Young, University of Science and Technology of China, No. 96 Jinzhai Road, Baohe District, Hefei, Anhui Province 230026, China}
\email{zsx0515@mail.ustc.edu.cn}

\author{Shanwei Zhang}
\address{Qiuzhen College, Tsinghua University, Beijing 100084, China}
\email{zsw26@mails.tsinghua.edu.cn}



\begin{abstract}
We study the Solomon--Terao algebra and the Solomon--Terao complex
of a central hyperplane arrangement $\mathcal A$ in $\mathbb C^\ell$.
Let $\eta$ be a homogeneous polynomial of degree $d\geq2$ whose
restriction to every flat of $\mathcal A$ has no critical point away from the origin.
We prove that the Solomon--Terao algebra $\mathrm{ST}(\mathcal A,\eta)$ has top degree
$|\mathcal A|+\ell(d-2)$ and a one-dimensional top graded component
with an explicit generator, solving a conjecture of Abe, Maeno,
Murai, and Numata.
For $d=2$, we show that the total dimension of $\mathrm{ST}(\mathcal A,\eta)$ is bounded below by
that of the Orlik--Solomon algebra.
Both results hold without assuming that the arrangement
$\mathcal A$ is tame.
Finally, we prove that $\mathcal A$ is tame if and only if
all positive-degree homology groups of the Solomon--Terao complex
associated with $\eta$ vanish.
\end{abstract}

\maketitle

\setcounter{tocdepth}{2}

\section{Introduction}
\label{sec:introduction}

The Solomon--Terao algebra and the Solomon--Terao complex relate the logarithmic derivations of hyperplane arrangements to their combinatorial and geometric invariants. The Solomon--Terao complex underlies the Solomon--Terao formula for the characteristic polynomial of an arrangement in terms of the Hilbert series of its logarithmic multiderivation modules \cite{SolomonTerao1987}, while Solomon--Terao algebras associated with certain arrangements realize coinvariant algebras of Weyl groups and cohomology rings of regular nilpotent Hessenberg varieties \cite[Section~1.2]{abe2019solomon}.

Abe, Maeno, Murai, and Numata formally introduced the Solomon--Terao algebra as the zeroth homology of the Solomon--Terao complex and proved that it is a complete intersection precisely when the arrangement is free \cite[Definition~1.6 and Theorem~1.10]{abe2019solomon}. For tame arrangements, its Hilbert series for suitable quadratic polynomials is given by a specialization of the Solomon--Terao polynomial \cite[Theorem~1.8]{abe2019solomon}. Tameness has thus played an important role in the study of these algebras beyond the free case. In this paper, we establish some basic results on their top degree and total dimension for general arrangements,
and show that the vanishing of positive-degree homology of the
Solomon--Terao complex characterizes tameness.

Throughout this introduction, let $\mathcal A$ be a central hyperplane arrangement in $\mathbb C^\ell$. We refer to Section~\ref{sec-pre} for the relevant definitions and notation. Abe, Maeno, Murai, and Numata conjectured that, for every $d\geq2$ and $\eta\in U_d(\mathcal A)$, the Solomon--Terao algebra $\mathrm{ST}(\mathcal A,\eta)$ has top degree $|\mathcal A|+\ell(d-2)$ and a one-dimensional top graded component \cite[Conjecture~5.12]{abe2019solomon}. Abe proved this conjecture for tame arrangements when $d=2$ and for essential tame arrangements when $d\geq2$ by regularity methods \cite[Theorems~1.7, 3.2, and~3.7]{abe2025solomon}. Our first theorem proves the conjecture for all $d\geq2$ without assuming tameness and explicitly identifies the top graded component.

\begin{theorem}\label{thm:main1}
For every $d\geq2$ and $\eta\in U_d(\mathcal A)$, the Solomon--Terao algebra $\mathrm{ST}(\mathcal A,\eta)$ has top degree $\tau:=|\mathcal A|+\ell(d-2)$, and its top graded component is one-dimensional with
\[
    \mathrm{ST}(\mathcal A,\eta)_\tau
    =\mathbb C\cdot[Q(\mathcal A)\det(\eta_{ij})],
\]
where $\eta_{ij}:=\partial_{x_i}\partial_{x_j}\eta$
for $1\le i,j\le\ell$.
\end{theorem}

The nonvanishing of this class was previously known for generic $\eta\in U_d(\mathcal A)$ \cite[Theorem~5.13]{abe2019solomon}. In particular, when $d=2$, the top degree is $|\mathcal A|$ and the top component is spanned by $[Q(\mathcal A)]$.

For tame arrangements, the total dimension of $\mathrm{ST}(\mathcal A,\eta)$ is known to equal that of the Orlik--Solomon algebra $\mathrm{OS}(\mathcal A)$ for every $\eta\in U_2(\mathcal A)$ \cite[Theorem~1.8]{abe2019solomon}. It is well-known that the latter is isomorphic to the cohomology ring of the arrangement complement \cite{OrlikSolomon1980}. Our second theorem extends this equality to a lower bound when $\mathcal A$ is not necessarily tame.

\begin{theorem}\label{thm:main2}
For every $\eta\in U_2(\mathcal A)$, we have
\[
    \dim_{\mathbb C}\mathrm{ST}(\mathcal A,\eta)\geq \dim_{\mathbb C}\mathrm{OS}(\mathcal A).
\]
\end{theorem}

Thus the total Betti number of the arrangement complement provides a uniform lower bound for $\dim_{\mathbb C}\mathrm{ST}(\mathcal A,\eta)$ for every $\eta\in U_2(\mathcal A)$\footnote{We note that parallel to our approach, a different proof of Theorem~\ref{thm:main2} will be given in \cite{Xiping}.}. When $\mathcal A$ is the complexification of a real arrangement, the right-hand side is also the number of real chambers \cite{Zaslavsky1975}. The inequality can be strict, while equality does not characterize tameness, as illustrated by Remarks~\ref{rem:ST-OS-strict} and~\ref{rem:ST-OS-equality-nontame}. For essential arrangements, similar arguments also yield
\[
    \dim_{\mathbb C}\mathrm{ST}(\mathcal A,\eta)\geq (-1)^\ell\chi_{\mathcal A}(1-d)
\]
for every $d\geq2$ and $\eta\in U_d(\mathcal A)$, as stated in Remark~\ref{rem:essential-higher-degree}.

To distinguish equality of total dimensions from the vanishing property that implies it, we return to the Solomon--Terao complex. Tameness implies that the positive-degree homology of the Solomon--Terao complex vanishes (known as ST-finiteness) \cite[Proposition~2.12]{abe2019solomon}. 
Our third theorem establishes the converse, answering the ST-finiteness part of \cite[Question~5.10]{abe2019solomon}.

\begin{theorem}\label{thm:main3}
Let $d\geq2$ and $\eta\in U_d(\mathcal A)$. Then
\[
    \mathcal A\text{ is tame}\quad\Longleftrightarrow\quad H_p(D^*(\mathcal A),\eta)=0\quad\text{for all }p>0.
\]
\end{theorem}

In particular, vanishing for one $d\geq2$ and one $\eta\in U_d(\mathcal A)$ implies vanishing for all $d\geq2$ and $\eta\in U_d(\mathcal A)$. Thus ST-finiteness, for these choices of $\eta$, depends only on the arrangement.

The relation between Theorems~\ref{thm:main2} and~\ref{thm:main3} can be made explicit by the Euler characteristic of the complex. For $\eta\in U_2(\mathcal A)$, the Solomon--Terao formula and the homological identity in \cite[Proposition~3.1]{abe2019solomon} give
\begin{equation}\label{eq:intro-dimension-excess}
    \dim_{\mathbb C}\mathrm{ST}(\mathcal A,\eta)-\dim_{\mathbb C}\mathrm{OS}(\mathcal A)=\sum_{p=1}^{\ell}(-1)^{p-1}\dim_{\mathbb C}H_p(D^*(\mathcal A),\eta).
\end{equation}
Theorem~\ref{thm:main2} asserts that this alternating sum in \eqref{eq:intro-dimension-excess} is always nonnegative, whereas Theorem~\ref{thm:main3} characterizes the simultaneous vanishing of all its terms. The example in Remark~\ref{rem:ST-OS-equality-nontame} indicates that nonzero homology contributions can cancel, so equality of dimensions is strictly weaker than ST-finiteness.

The proofs of our three main results use three different mechanisms. For Theorem~\ref{thm:main1}, we combine Abe's Castelnuovo--Mumford regularity bound \cite[Theorem~1.10]{abe2025solomon} with a local-cohomology spectral sequence to isolate the top graded component. Comparison with the empty arrangement identifies its generator. For Theorem~\ref{thm:main2}, after reducing to the essential case, we apply Huh's theorem for smooth very affine varieties \cite[Theorem~1]{huh2013maximum} to the hypersurface $\{\eta=1\}$ in the arrangement complement. A computation of its Euler characteristic identifies $\dim_{\mathbb C}\mathrm{OS}(\mathcal A)$ with the number of zeros of a generic logarithmic $1$-form on this hypersurface. We then compare this count with $\dim_{\mathbb C}\mathrm{ST}(\mathcal A,\eta)$ through an auxiliary ideal and a comparison of highest-degree forms. For Theorem~\ref{thm:main3}, we tensor the Solomon--Terao complex with the Koszul complex on the partial derivatives of $\eta$. 
Under the vanishing assumption, a chain isomorphism yields a direct-sum decomposition of Tor groups. Nonvanishing of Tor against a nonzero finite-length module then detects the projective dimensions of the modules $D^p(\mathcal A)$ and hence establishes tameness.

\subsection*{Acknowledgments}
We are grateful to Yongqiang Liu and Xiping Zhang for valuable discussions. This research was initiated during the PKU Algebraic Combinatorics Experience (PACE 2026) at Beijing International Center for Mathematical Research, Peking University. We are deeply grateful to our mentors Yibo Gao and Tao Gui for their guidance, insightful discussions, and helpful observations. We also thank the organizers of PACE for providing this research opportunity. This work was partially done while the authors were visiting the Shanghai Institute for Mathematics and Interdisciplinary Sciences (SIMIS). We thank Dingxin Zhang for the hospitality.

\subsection*{AI disclosure}
Generative AI was used interactively in the mathematical exploration, including developing, testing, and refining the proofs, and as a technical assistant for checking calculations and arguments, locating relevant references, improving the exposition and identifying mathematical and typographical errors. All AI-assisted content was carefully verified by the authors, who take full responsibility for the manuscript.

\section{Preliminaries} \label{sec-pre}

\subsection{Hyperplane arrangements and Orlik--Solomon algebras}
\label{subsec:arrangements-OS}

We recall some basic definitions and properties of
hyperplane arrangements and Orlik--Solomon algebras,
following \cite{orlik1992arrangements,stanley2007introduction}.

Let $V=\mathbb C^\ell$.
A \emph{hyperplane arrangement} $\mathcal A$ in $V$
is a finite set of hyperplanes. It is called \emph{central}
if all its hyperplanes contain the origin.
Throughout this paper, all hyperplane arrangements are assumed
to be central.
The empty arrangement in $V$ is denoted by $\Phi_V$,
and the \emph{complement} of $\mathcal A$ is
\[
    U(\mathcal A):=V\setminus\bigcup_{H\in\mathcal A}H.
\]

\begin{definition}
The \emph{intersection lattice} $L(\mathcal A)$ consists
of all intersections of hyperplanes in $\mathcal A$,
called \emph{flats}, ordered by reverse inclusion.
Explicitly,
\[
    L(\mathcal A)
    :=
    \left\{
        \bigcap_{H\in\mathcal B}H
        \;\middle|\;
        \mathcal B\subseteq\mathcal A
    \right\},
\]
where the empty intersection is $V$.
\end{definition}

The \emph{center} and \emph{rank} of $\mathcal A$ are
\[
    T:=T(\mathcal A)=\bigcap_{H\in\mathcal A}H,
    \qquad
    r:=r(\mathcal A)=\operatorname{codim}_V T.
\]
The hyperplane arrangement $\mathcal A$ is called
\emph{essential} if $T=\{0\}$, or equivalently, if $r=\ell$.
Its \emph{essentialization} is
$\mathcal A^{\mathrm{ess}}:=\{H/T\mid H\in\mathcal A\}$
in $V/T$. Choosing a decomposition $V=U\oplus T$,
we identify $\mathcal A^{\mathrm{ess}}$ with
$\{H\cap U\mid H\in\mathcal A\}$ in $U$.

For hyperplane arrangements $\mathcal A_i$ in $V_i$ ($i=1,2$),
their \emph{product} is
\[
    \mathcal A_1\times\mathcal A_2
    :=
    \{H_1\oplus V_2\mid H_1\in\mathcal A_1\}
    \cup
    \{V_1\oplus H_2\mid H_2\in\mathcal A_2\}.
\]
Since $H=(H\cap U)\oplus T$ for every $H\in\mathcal A$,
we have 
\(\mathcal A=\mathcal A^{\mathrm{ess}}\times\Phi_T\).

\begin{definition}
The \emph{Möbius function}
$\mu:L(\mathcal A)\to\mathbb Z$ is defined recursively by
\[
    \mu(V)=1,\qquad
    \mu(X)=-\sum_{\substack{Y\in L(\mathcal A)\\X\subsetneq Y}}
    \mu(Y)
    \quad (X\neq V).
\]
The \emph{characteristic polynomial} and
\emph{Poincar\'e polynomial} of $\mathcal A$ are
defined by
\begin{align}
    \chi_{\mathcal A}(t)
    &:=\sum_{X\in L(\mathcal A)}
    \mu(X)t^{\dim X},\notag\\
    \pi_{\mathcal A}(t)
    &:=\sum_{X\in L(\mathcal A)}
    \mu(X)(-t)^{\operatorname{codim}_V X}.
    \label{eq-Poin}
\end{align}
\end{definition}

These polynomials satisfy
$\pi_{\mathcal A}(t)=(-t)^\ell\chi_{\mathcal A}(-t^{-1})$.
Moreover, $\chi_{\mathcal A}(1)=0$ whenever $r>0$,
by the defining recurrence for $\mu$.

\vspace{5pt}

Fix a total order on $\mathcal A$, and let
\[
    E(\mathcal A):=
    \bigwedge_{\mathbb C}
    \left(\bigoplus_{H\in\mathcal A}\mathbb C e_H\right).
\]
For $\mathcal B=\{H_1<\cdots<H_p\}\subseteq\mathcal A$,
write $e_{\mathcal B}:=e_{H_1}\wedge\cdots\wedge e_{H_p}$,
with $e_\varnothing=1$.
Define the $\mathbb C$-linear boundary operator
$\partial_{\mathrm{OS}}:E(\mathcal A)\to E(\mathcal A)$ by
\[
    \partial_{\mathrm{OS}}e_{\mathcal B}
    :=\sum_{i=1}^p(-1)^{i-1}
    e_{\mathcal B\setminus\{H_i\}},
    \qquad \partial_{\mathrm{OS}}1=0.
\]
A subset $\mathcal B\subseteq\mathcal A$ is
\emph{dependent} if
\[
    \operatorname{codim}_V
    \bigcap_{H\in\mathcal B}H<|\mathcal B|.
\]

\begin{definition}
Let $I(\mathcal A)$ be the ideal of $E(\mathcal A)$
generated by $\partial_{\mathrm{OS}}e_{\mathcal B}$
for all dependent subsets $\mathcal B\subseteq\mathcal A$.
The \emph{Orlik--Solomon algebra} of $\mathcal A$ is
\[
    \mathrm{OS}(\mathcal A):=E(\mathcal A)/I(\mathcal A).
\]
\end{definition}

The Orlik--Solomon algebra is naturally graded by
$\deg e_H=1$ for each $H\in\mathcal A$.
For a graded $\mathbb C$-vector space
$M=\bigoplus_{k\in\mathbb Z}M_k$
that is bounded below and has finite-dimensional
graded components, its \emph{Hilbert series} is
\[
    \operatorname{Hilb}(M;x)
    :=\sum_{k\in\mathbb Z}
    (\dim_{\mathbb C}M_k)x^k.
\]

\begin{theorem}[{\cite[Theorem~5.2]{OrlikSolomon1980}}]
Let $\mathcal A$ be a complex hyperplane arrangement.
There is an isomorphism of graded $\mathbb C$-algebras
\(\mathrm{OS}(\mathcal A) \cong H^*(U(\mathcal A);\mathbb C)\).
Moreover,
\begin{equation}\label{eq:OS-Hilbert-Poincare}
    \operatorname{Hilb}(\mathrm{OS}(\mathcal A);t)
    =\pi_{\mathcal A}(t).
\end{equation}
\end{theorem}

In particular,
\(\dim_{\mathbb C}\mathrm{OS}(\mathcal A)
=\pi_{\mathcal A}(1)=(-1)^\ell\chi_{\mathcal A}(-1)\).

\vspace{1pt}
\subsection{Solomon--Terao complex and Solomon--Terao algebra} 

We introduce some basic definitions about logarithmic derivation modules $D^p(\mathcal{A})$ for a hyperplane arrangement $\mathcal{A}$, the Solomon--Terao complex and the Solomon--Terao algebra, mainly following \cite{abe2019solomon} and \cite{SolomonTerao1987}.

For a graded ring $R$ and a graded module $M$, 
we write $R_d$ and $M_d$ for their homogeneous components of degree \(d\), respectively.
We use the convention for \emph{graded shift} that $M(a)_d=M_{a+d}$ for every $d\in\mathbb Z$.
The \emph{end} of $M$ is defined by
\[
    \operatorname{end}M 
    :=
    \sup\{d\in\mathbb Z\mid M_d\neq0\},
\]
with the convention $\operatorname{end}0=-\infty$.

\vspace{5pt}

For any complex vector space $W$,
write $S_W:=\operatorname{Sym}(W^*)$.
In particular, $S:=S_V=\mathbb C[x_1,\ldots,x_\ell]$.
The module of $\mathbb C$-derivations of $S$ is 
\[
    \operatorname{Der}S
    :=\bigoplus_{i=1}^{\ell}S\partial_{x_i},
\]
graded by $\deg x_i=1$ and $\deg\partial_{x_i}=0$.
For $1\le p\le\ell$, set
$\operatorname{Der}^pS:=\bigwedge_S^p\operatorname{Der}S$,
with $\operatorname{Der}^0S=S$.

For each $H\in\mathcal A$, fix
$\alpha_H\in V^*$ with $\ker\alpha_H=H$, and set
\[
    Q(\mathcal A):=\prod_{H\in\mathcal A}\alpha_H.
\]
For $1\le p\le\ell$, the module of
\emph{logarithmic $p$-derivations} of $\mathcal A$ is
\[
    D^p(\mathcal A)
    :=
    \left\{
        \theta\in\operatorname{Der}^pS
        \;\middle|\;
        \begin{array}{l}
        \theta(\alpha_H,f_2,\ldots,f_p)\in\alpha_HS\\
        \text{for all }H\in\mathcal A,\ 
        f_2,\ldots,f_p\in S
        \end{array}
    \right\}.
\]
Set $D^0(\mathcal A)=S$ and
$D^p(\mathcal A)=0$ for $p<0$ or $p>\ell$.
In particular,
\[
    D^1(\mathcal A)
    =\{\theta\in\operatorname{Der}S
      \mid \theta(\alpha_H)\in\alpha_HS
      \text{ for all }H\in\mathcal A\}.
\]
Note that $D^p(\mathcal A)\neq0$ for $0\le p\le\ell$, since 
$Q(\mathcal A)\operatorname{Der}^pS\subseteq D^p(\mathcal A)$.
In the top degree,
\begin{equation}\label{eq:top-derivation}
    D^\ell(\mathcal A)
    =Q(\mathcal A)\operatorname{Der}^\ell S
    \cong S(-|\mathcal A|).
\end{equation}

The following is an immediate consequence of
\cite[Proposition~5.8]{SolomonTerao1987}.

\begin{lemma}
\label{lem:product-log-derivations}
Let $V=U\oplus T$, where $T=T(\mathcal A)$,
and identify $\mathcal A^{\mathrm{ess}}$ with
a hyperplane arrangement in $U$.
Then, for $0\le p\le\ell$, there is an isomorphism
of graded $S$-modules
\[
    D^p(\mathcal A)
    \cong
    \bigoplus_{i+j=p}
    D^i(\mathcal A^{\mathrm{ess}})
    \otimes_{\mathbb C}
    \operatorname{Der}^j S_T.
\]
\end{lemma}

\begin{definition}
Let $d\ge0$ and $\eta\in S_d$.
For $1\le p\le\ell$, define the $S$-linear map
$\partial_p:D^p(\mathcal A)\to D^{p-1}(\mathcal A)$ by
\[
    (\partial_p\theta)(f_2,\ldots,f_p)
    :=\theta(\eta,f_2,\ldots,f_p)
\]
for $\theta\in D^p(\mathcal A)$ and
$f_2,\ldots,f_p\in S$.
Set $\partial_0=0$.
By alternatingness, $\partial_{p-1}\circ\partial_p=0$,
giving the \emph{Solomon--Terao complex}
\[
    0\to D^\ell(\mathcal A)
    \xrightarrow{\partial_\ell}
    \cdots\xrightarrow{\partial_2}
    D^1(\mathcal A)
    \xrightarrow{\partial_1}S\to0.
\]
Each $\partial_p$ is homogeneous of degree $d-1$.
Its $p$-th homology is the graded $S$-module
\[
    H_p(D^*(\mathcal A),\eta)
    :=\ker\partial_p/\operatorname{im}\partial_{p+1}.
\]
The \emph{Solomon--Terao ideal} is
$\mathfrak a(\mathcal A,\eta):=\operatorname{im}\partial_1
=\{\theta(\eta)\mid\theta\in D^1(\mathcal A)\}$.
The \emph{Solomon--Terao algebra} is
\[
    \mathrm{ST}(\mathcal A,\eta)
    :=H_0(D^*(\mathcal A),\eta)
    =S/\mathfrak a(\mathcal A,\eta).
\]
\end{definition}

\vspace{5pt}

For $X\in L(\mathcal A)$ and $f\in S_X$,
define the \emph{Jacobian ideal} of $f$ by
\[
    \operatorname{Jac}_X(f)
    :=(\theta(f)\mid\theta\in\operatorname{Der}S_X).
\]
Following \cite[Definition~2.7]{abe2019solomon},
a homogeneous polynomial $f\in S_X$ is called
\emph{nondegenerate on $X$} if
$\operatorname{Jac}_X(f)$ has no zeros in
$X\setminus\{0\}$.
For $d>0$, define
\[
    U_d(\mathcal A)
    :=
    \left\{
        \eta\in S_d\;\middle|\;
        \eta|_X\text{ is nondegenerate on }X
        \text{ for all }X\in L(\mathcal A)
    \right\}.
\]
We call $\eta\in U_d(\mathcal A)$
\emph{$\mathcal A$-nondegenerate}.

\begin{proposition}[{\cite[Proposition~2.8(1)]{abe2019solomon}}]
\label{prop:ST-finite-homology}
For every $d>0$, the set $U_d(\mathcal A)$ is a
nonempty Zariski-open subset of $S_d$.
Moreover, for every $\eta\in U_d(\mathcal A)$,
the homology groups $H_p(D^*(\mathcal A),\eta)$
are finite-dimensional over $\mathbb C$
for all $0\le p\le\ell$.
\end{proposition}

\subsection{Tame arrangements and Solomon--Terao polynomials}

An arrangement $\mathcal A$ is called \emph{tame} if
\[
    \operatorname{pd}_S D^p(\mathcal A)\le\ell-p
    \qquad (1\le p\le\ell),
\]
where $\operatorname{pd}_S$ denotes projective dimension over $S$.

\begin{proposition}[{\cite[Proposition~2.12]{abe2019solomon}}]
\label{prop-vanish}
If $\mathcal A$ is tame, then for every $d>0$ and
$\eta\in U_d(\mathcal A)$,
\begin{equation}\label{eq-ST finite}
    H_p(D^*(\mathcal A),\eta)=0
    \qquad (p>0).
\end{equation}
\end{proposition}

For $d>0$ and $\eta\in U_d(\mathcal A)$, we say that
$\mathcal A$ is \emph{ST-finite with respect to $\eta$}
if \eqref{eq-ST finite} holds.
This extends the terminology of
\cite[Definition~5.7(4)]{abe2019solomon}
to arrangements not assumed to be tame.

\begin{definition}
The \emph{Solomon--Terao bi-polynomial} is defined by
\begin{equation}\label{eq-ST bipoly}
    \Psi(\mathcal A;x,t)
    :=\sum_{p=0}^{\ell}
    \operatorname{Hilb}(D^p(\mathcal A);x)
    \bigl(t(x-1)-1\bigr)^p.
\end{equation}
We follow the normalization of
\cite[Definition~1.4]{abe2025solomon}.\footnote{If
$\widetilde\Psi$ denotes the polynomial in
\cite[Definition~1.2]{abe2019solomon}, then
$\widetilde\Psi(\mathcal A;x,t)
=t^\ell\Psi(\mathcal A;x,-t^{-1})$.
In particular, $\widetilde\Psi(\mathcal A;x,1)
=\Psi(\mathcal A;x,-1)$.}
\end{definition}

\begin{theorem}[{\cite[Theorem~1.2 and Proposition~5.3]{SolomonTerao1987}}]
\label{thm:ST-formula}
The rational function $\Psi(\mathcal A;x,t)$
is a polynomial in $x$ and $t$, satisfying
\[
    \Psi(\mathcal A;1,t)=(-1)^\ell\chi_{\mathcal A}(t).
\]
\end{theorem}

For $\eta\in U_2(\mathcal A)$, applying the graded
Euler--Poincar\'e formula to the shifted complex
with terms $D^p(\mathcal A)(-p)$ yields the following.

\begin{proposition}[{\cite[Proposition~3.1]{abe2019solomon}}]
\label{prop-coho}
For every $\eta\in U_2(\mathcal A)$,
\[
    \Psi(\mathcal A;x,-1)
    =\sum_{p=0}^{\ell}(-x)^p
    \operatorname{Hilb}\bigl(H_p(D^*(\mathcal A),\eta);x\bigr).
\]
\end{proposition}

The following is \cite[Theorem~1.8]{abe2019solomon},
restated in our normalization.

\begin{corollary}\label{coro-ST=OS}
If $\mathcal A$ is tame, then for every
$\eta\in U_2(\mathcal A)$,
\[
    \operatorname{Hilb}(\mathrm{ST}(\mathcal A,\eta);x)
    =\Psi(\mathcal A;x,-1).
\]
In particular,
\[
    \dim_{\mathbb C}\mathrm{ST}(\mathcal A,\eta)
    =\dim_{\mathbb C}\mathrm{OS}(\mathcal A).
\]
\end{corollary}

\section{Top-degree of Solomon--Terao algebra} 

In this section, we prove Theorem~\ref{thm:main1}.
The proof is organized as follows.
We first establish some preliminary results and then use them to compute the local cohomology modules of $D^p(\mathcal A)$.
We then construct a double complex and its two spectral sequences to determine the top graded component of $\mathrm{ST}(\mathcal A,\eta)$.
Finally, we compare with the empty arrangement via multiplication by $Q(\mathcal A)$ to identify $[Q(\mathcal A)\det(\eta_{ij})]$ as a generator of this component.

The interplay between local cohomology, logarithmic differential forms, and homological properties of hyperplane arrangements has also been investigated in \cite{DSSWW2013}. For background on local cohomology and \v Cech complexes, we refer to Brodmann and Sharp \cite[Chapters~1 and~5]{brodmannSharp2012}.
Let $R=\mathbb C[y_1,\dots,y_n]$ be the standard graded polynomial ring over $\mathbb C$.
For an ideal $\mathfrak a=(a_1,\dots,a_m)\subseteq R$ and an $R$-module $M$,
we write $H^q_{\mathfrak a}(M)$ for the $q$-th \emph{local cohomology} module of $M$ with respect to $\mathfrak a$,
with the convention that $H^q_{\mathfrak a}(M)=0$ for $q<0$.
Let $C(M)^\bullet$ denote the \emph{\v Cech complex} of $M$ with respect to the generators $a_1,\ldots,a_m$.
It is well known that the cohomology of the \v Cech complex computes the local cohomology, giving
\[
    H^q(C(M)^\bullet)\cong H^q_{\mathfrak a}(M)
\]
for every $q\in\mathbb Z$.
We write $R_+:=\bigoplus_{d>0}R_d$.

We record the following direct consequence of the K\"unneth formula for complexes
\cite[Theorem~3.6.3]{Wei94}, applied to \v Cech complexes.

\begin{lemma}
\label{kunneth}
Let $R_1$ and $R_2$ be standard graded polynomial rings over $\mathbb C$,
and set $R=R_1\otimes_{\mathbb C}R_2$ with the usual total grading.
For a graded $R_1$-module $M$ and a graded $R_2$-module $N$,
there is a natural isomorphism of graded $R$-modules
\[
    H^q_{R_+}(M\otimes_{\mathbb C}N)
    \cong
    \bigoplus_{i+j=q}
    H^i_{(R_1)_+}(M)\otimes_{\mathbb C}H^j_{(R_2)_+}(N)
\]
for every $q\in\mathbb Z$.
\end{lemma}

We record the following computation of local cohomology for polynomial rings.

\begin{lemma}
\label{polyring}
Let $n\ge0$ and let
$R=\mathbb C[y_1,\ldots,y_n]$ be standard graded, with
$R_+=(y_1,\ldots,y_n)$.
Then
\[
    H^q_{R_+}(R)=0
    \qquad\text{for }q\neq n,
\]
while
\[
    \operatorname{end}H^n_{R_+}(R)=-n,
    \qquad
    \dim_{\mathbb C}H^n_{R_+}(R)_{-n}=1.
\]
\end{lemma}

\begin{proof}
For $n=0$, the assertions are immediate.
For $n\ge1$, the one-variable \v Cech complex gives
$H^q_{(y_i)}(\mathbb C[y_i])=0$ for $q\ne1$ and
$H^1_{(y_i)}(\mathbb C[y_i])
\cong\bigoplus_{d_i\ge1}\mathbb C y_i^{-d_i}$.
Applying Lemma~\ref{kunneth} to
$R=\mathbb C[y_1]\otimes_{\mathbb C}\cdots\otimes_{\mathbb C}\mathbb C[y_n]$
gives
\[
    H^q_{R_+}(R)=0\quad(q\ne n),
    \qquad
    H^n_{R_+}(R)
    \cong
    \bigoplus_{d_1,\ldots,d_n\ge1}
    \mathbb C y_1^{-d_1}\cdots y_n^{-d_n}.\qedhere
\]
\end{proof}

We adopt the following definition of Castelnuovo--Mumford regularity from \cite[Definition~16.2.9]{brodmannSharp2012}.

\begin{definition}
\label{reg}
For a finitely generated graded $R$-module $M$, its
\emph{Castelnuovo--Mumford regularity} is
\[
    \operatorname{reg}(M)
    :=
    \sup_{q\ge0}
    \left\{
        \operatorname{end} H^q_{R_+}(M)+q
    \right\}.
\]
\end{definition}
For the remainder of this section, we set $V=\mathbb C^\ell$ and $d\ge2$.
Fix a central hyperplane arrangement $\mathcal A$ in $V$ and $\eta\in U_d(\mathcal{A})$.
We write $\tau:=|\mathcal A|+\ell(d-2)$.

We next compute the local cohomology of $D^p(\mathcal A)$ using the regularity bound for logarithmic multiderivations.

\begin{proposition}\label{prop:HqDp}
For every $0\le p\le q\le\ell$ and
$u\ge \tau$, we have
\[
\dim_{\mathbb C} H^q_{S_+}(D^p(\mathcal A))_{u-p(d-1)}
=
\begin{cases}
1, & \text{if } (p,q,u)=(\ell,\ell,\tau),\\
0, & \text{otherwise}.
\end{cases}
\]
\end{proposition}

\begin{proof}
We use Lemma~\ref{lem:product-log-derivations} and Lemma~\ref{kunneth} to express the dimension of the local cohomology as a sum, 
then apply the regularity bound to eliminate all summands except the one with $i=r$.
We write $v_p:=u-p(d-1)$. Then

\begin{equation}\label{STbound1}
\begin{aligned}
    \dim_{\mathbb C} H^q_{S_+}(D^p(\mathcal A))_{v_p}
    &=\dim_\mathbb C H_{S_+}^q(\bigoplus_{i+j=p}D^i(\mathcal{A}^\mathrm{ess})\otimes_\mathbb C \operatorname{Der}^j S_T)_{v_p}\\
    &=\sum_{i+j=p}\dim_\mathbb C H_{S_+}^q(D^i(\mathcal{A}^\mathrm{ess})\otimes_\mathbb C \operatorname{Der}^j S_T)_{v_p}\\
    &=\sum_{i+j=p}\binom{\ell-r}{j}\dim_\mathbb C H_{S_+}^q(D^i(\mathcal{A}^\mathrm{ess})\otimes_\mathbb C S_T)_{v_p}\\
    &=\sum_{i+j=p}\binom{\ell-r}{j}\dim_\mathbb C (H_{(S_U)_+}^{q-\ell+r}(D^i(\mathcal{A}^\mathrm{ess}))\otimes_\mathbb C H_{(S_T)_+}^{\ell-r}(S_T))_{v_p}\\
\end{aligned}
\end{equation}
The first equality follows from Lemma~\ref{lem:product-log-derivations}.
The second equality holds by the additivity of local cohomology.
The third equality uses that $\operatorname{Der}^j S_T$ is a free $S_T$-module of rank $\binom{\ell-r}{j}$.
The last equality comes from Lemma~\ref{kunneth} and Lemma~\ref{polyring}.

By \cite[Theorem~1.10]{abe2025solomon}, we have  
$\operatorname{reg}D^i(\mathcal A^{\mathrm{ess}}) \le |\mathcal A|-r+i$ for $i<r$.
Combining this with Definition~\ref{reg} and Lemma~\ref{polyring}, we have
\begin{align*}
&\,\operatorname{end}\left(H_{(S_U)_+}^{q-\ell+r}(D^i(\mathcal A^{\mathrm{ess}}))
\otimes_{\mathbb C}H_{(S_T)_+}^{\ell-r}(S_T)\right)\\
=&\,\operatorname{end}H_{(S_U)_+}^{q-\ell+r}(D^i(\mathcal A^{\mathrm{ess}}))+\operatorname{end}H_{(S_T)_+}^{\ell-r}(S_T)\\
\le&\,|\mathcal A|-r+i-(q-\ell+r)-(\ell-r)\\
\le&\, v_p-(\ell-p)(d-2)-(r-i)<v_p.
\end{align*}
Thus all summands with $i<r$ in the last row of \eqref{STbound1} vanish, so \eqref{STbound1} reduces to

\begin{align*}
    \dim_{\mathbb C} H^q_{S_+}(D^p(\mathcal A))_{v_p}
    &=\binom{\ell-r}{p-r}\dim_\mathbb C \left(H_{(S_U)_+}^{q-\ell+r}(D^r(\mathcal{A}^\mathrm{ess}))\otimes_\mathbb C H_{(S_T)_+}^{\ell-r}(S_T)\right)_{v_p}\\
    &=\binom{\ell-r}{p-r}\dim_\mathbb C \left(H_{(S_U)_+}^{q-\ell+r}(S_U)\otimes_\mathbb C H_{(S_T)_+}^{\ell-r}(S_T)\right)_{v_p-|\mathcal{A}|}\\
    &=\begin{cases}
        1, & \text{if } (p,q,u)=(\ell,\ell,\tau),\\
        0, & \text{otherwise}.
    \end{cases}
\end{align*}
The second equality comes from \eqref{eq:top-derivation} and
the fact that local cohomology commutes with graded shifts.
The last equality follows from Lemma~\ref{polyring}.
\end{proof}

We next combine Proposition~\ref{prop:HqDp} with the spectral sequence method to determine the top graded components of the Solomon--Terao algebra.

\begin{proposition}
\label{double-complex}
For every $u\ge \tau$, there is an isomorphism
\[
    \mathrm{ST}(\mathcal A,\eta)_u
    \cong
    H^\ell_{S_+}(D^\ell(\mathcal A))_{u-\ell(d-1)}.
\]
In particular,
\[
    \dim_{\mathbb C}\mathrm{ST}(\mathcal A,\eta)_u=
    \begin{cases}
        0,&u>\tau,\\
        1,&u=\tau.
    \end{cases}
\]
\end{proposition}

\begin{proof}
Let $K^\bullet$ be the cochain complex concentrated in degrees $-\ell,\ldots,0$,
with $K^{-p}=D^p(\mathcal A)(-p(d-1))$ for $0\le p\le\ell$, 
and differentials induced by $\partial_p$.
Since $\partial_p$ is homogeneous of degree $d-1$, 
these differentials are homogeneous of degree zero.
By definition, $H^0(K^\bullet)=\mathrm{ST}(\mathcal A,\eta)$.

Set $C^{p,q}:=C(K^p)^q$, where $C(K^p)^\bullet$ is the \v Cech complex of $K^p$ with respect to $x_1,\ldots,x_\ell$.
Since the construction $M\mapsto C(M)^\bullet$ is functorial,
the differentials of $K^\bullet$ induce chain maps between the corresponding \v Cech complexes.
This gives a bounded double cochain complex $C^{\bullet,\bullet}$ of graded $S$-modules, 
with all differentials homogeneous of degree zero.

We consider two spectral sequences associated with $C^{\bullet,\bullet}$. 
Let ${}^{I}E$ and ${}^{II}E$ denote the spectral sequences obtained by first taking cohomology of $C^{p,\bullet}$ and $C^{\bullet,p}$, respectively, with $E_1$-pages
\[
\begin{aligned}
    {}^{I}E_1^{p,q}
    &=H^q(C^{p,\bullet})\cong H^q_{S_+}(K^p),\\
    {}^{II}E_1^{p,q}
    &=H^q(C^{\bullet,p}).
\end{aligned}
\]

As a finite direct sum of localization functors, each $C(-)^q$ is exact, so
${}^{II}E_1^{p,q}=H^q(C(K^\bullet)^p) \cong C(H^q(K^\bullet))^p$, and then
${}^{II}E_2^{p,q}\cong H^p_{S_+}(H^q(K^\bullet))$.
By Proposition~\ref{prop:ST-finite-homology}, the graded modules $H^q(K^\bullet)$ are finite-dimensional over $\mathbb C$, and hence $S_+$-torsion.
Applying \cite[Corollary~2.1.7(i)]{brodmannSharp2012} to $H^q(K^\bullet)$, we obtain
\[
    {}^{II}E_2^{p,q}
    \cong H^p_{S_+}\!\left(H^q(K^\bullet)\right)
    \cong
    \begin{cases}
        H^q(K^\bullet), & p=0,\\
        0, & p\ne0.
    \end{cases}
\]
Thus ${}^{II}E$ collapses at $E_2$, 
and $H^0(\operatorname{Tot}(C))\cong {}^{II}E_2^{0,0}\cong \mathrm{ST}(\mathcal A,\eta)$ as graded $S$-modules.

Fix $u\ge\tau$.
For $0\le p,q\le\ell$, the spectral sequence ${}^{I}E$ satisfies
\[
    \left({}^{I} E_1^{-p,q}\right)_u
    \cong
    H^q_{S_+}\!\left(D^p(\mathcal A)\right)_{u-p(d-1)}.
\]
For every $n\ge1$, the differentials involving $\left({}^{I} E_n^{-\ell,\ell}\right)_u$ have the form
\[
    \left({}^{I} E_n^{-\ell-n,\ell+n-1}\right)_u
    \longrightarrow
    \left({}^{I} E_n^{-\ell,\ell}\right)_u
    \longrightarrow
    \left({}^{I} E_n^{-\ell+n,\ell-n+1}\right)_u.
\]
The term on the left is zero since $-\ell-n<-\ell$ and
the term on the right is zero by Proposition~\ref{prop:HqDp}.
Thus we have the isomorphism $\left({}^{I} E_1^{-\ell,\ell}\right)_u\cong\left({}^{I} E_\infty^{-\ell,\ell}\right)_u$.
By Proposition~\ref{prop:HqDp} again,
$\left({}^{I} E_\infty^{-p,p}\right)_u=\left({}^{I} E_1^{-p,p}\right)_u=0$ for every $0\le p<\ell$.
Hence
\[
\mathrm{ST}(\mathcal A,\eta)_u
\cong H^0(\operatorname{Tot}(C))_u
\cong \left({}^{I} E_\infty^{-\ell,\ell}\right)_u
\cong \left({}^{I} E_1^{-\ell,\ell}\right)_u
\cong H^\ell_{S_+}\!\left(D^\ell(\mathcal A)\right)_{u-\ell(d-1)}.
\]
By Proposition~\ref{prop:HqDp},
$H^\ell_{S_+}(D^\ell(\mathcal A))_{u-\ell(d-1)}$ has dimension
$0$ if $u>\tau$ and dimension $1$ if $u=\tau$.
This proves the proposition.
\end{proof}

We next identify the top graded component explicitly by comparing the spectral sequences of $\mathcal A$ and $\Phi_V$, which are constructed in the proof of Proposition~\ref{double-complex}.

\begin{proposition}
\label{Q(A)det}
We have
\[
    \mathrm{ST}(\mathcal A,\eta)_\tau
    =\mathbb C\cdot[Q(\mathcal A)\det(\eta_{ij})].
\]
\end{proposition}
\begin{proof}
Let $\Phi_V$ be the empty arrangement in $V$. Since $V\in L(\mathcal A)$,
we have $\eta\in U_d(\Phi_V)$. Since $\Phi_V$ is free,
$\mathrm{ST}(\Phi_V,\eta)$ is a complete intersection by
\cite[Theorem~1.10]{abe2019solomon}.
Hence \cite[Proposition~3.2]{abe2019solomon} shows that
\begin{equation}
    [\det(\eta_{ij})]\ne0
    \quad\text{in }\mathrm{ST}(\Phi_V,\eta).
    \label{nonzero}
\end{equation}

For every $p$, since $D^p(\Phi_V)=\operatorname{Der}^p S$ and
$Q(\mathcal A)\operatorname{Der}^p S\subseteq D^p(\mathcal A)$,
multiplication by $Q(\mathcal A)$ defines a map
$D^p(\Phi_V)\to D^p(\mathcal A)$.
For $\theta\in D^p(\Phi_V)$, the $S$-linearity of $\partial_p$ gives
$\partial_p(Q(\mathcal A)\theta)=Q(\mathcal A)\partial_p(\theta)$.
Thus multiplication by $Q(\mathcal A)$ defines a homogeneous chain map
of degree $|\mathcal A|$ from the Solomon--Terao complex $(D^\bullet(\Phi_V),\partial_\bullet)$
to $(D^\bullet(\mathcal A),\partial_\bullet)$.

Since $D^\ell(\mathcal A)=Q(\mathcal A)\operatorname{Der}^\ell S$,
multiplication by $Q(\mathcal A)$ induces an isomorphism
\[
    H^\ell_{S_+}(D^\ell(\Phi_V))_{-\ell}
    \xrightarrow{\cdot Q(\mathcal A)}
    H^\ell_{S_+}(D^\ell(\mathcal A))_{|\mathcal A|-\ell}.
\]
The chain map above induces a morphism between the corresponding double complexes and hence morphisms between the associated spectral sequences.
The isomorphisms in Proposition~\ref{double-complex} are therefore compatible with multiplication by $Q(\mathcal A)$.
Thus the induced map
\[
    \mathrm{ST}(\Phi_V,\eta)_{\ell(d-2)}
    \xrightarrow{\cdot Q(\mathcal A)}
    \mathrm{ST}(\mathcal A,\eta)_{|\mathcal A|+\ell(d-2)}
\]
is also an isomorphism.
Hence $[Q(\mathcal A)\det(\eta_{ij})]\neq0$ in $\mathrm{ST}(\mathcal A,\eta)$ by \eqref{nonzero}.
Combining this with $\dim_{\mathbb C}\mathrm{ST}(\mathcal A,\eta)_\tau=1$ from Proposition~\ref{double-complex}, we obtain
\[
    \mathrm{ST}(\mathcal A,\eta)_\tau
    =\mathbb C\cdot[Q(\mathcal A)\det(\eta_{ij})].\qedhere
\]
\end{proof}

Finally, Proposition~\ref{double-complex} and Proposition~\ref{Q(A)det} immediately imply Theorem~\ref{thm:main1}.

\section{An Orlik--Solomon bound for the Solomon--Terao algebra}
\label{sec:ST-OS-bound}

In this section, we prove Theorem~\ref{thm:main2}.
Let $\mathcal A$ be a central hyperplane arrangement in
$V=\mathbb C^\ell$, and fix $\eta\in U_2(\mathcal A)$.

The proof is organized as follows.
We first reduce to the case where $\mathcal A$ is essential of
positive rank. We then use a logarithmic one-form with exactly
$\dim_{\mathbb C}\mathrm{OS}(\mathcal A)$ zeros to construct an ideal
$\mathfrak a_1\subseteq S$ whose zero locus contains these points.
The Chinese remainder theorem gives
$\dim_{\mathbb C}S/\mathfrak a_1\geq
\dim_{\mathbb C}\mathrm{OS}(\mathcal A)$, while comparison with the
degree form ideal of $\mathfrak a_1$ gives
$\dim_{\mathbb C}S/\mathfrak a_1\leq
\dim_{\mathbb C}\mathrm{ST}(\mathcal A,\eta)$, proving the theorem.

For quadratic forms, the definition of $\mathcal A$-nondegeneracy
has the following linear-algebraic reformulation.
For $f\in S_2$, let
\[
    B_f(u,v)
    :=
    \frac{1}{2}
    \bigl(f(u+v)-f(u)-f(v)\bigr)
\]
be its polar symmetric bilinear form.  Thus
\begin{equation}
    f(x)=B_f(x,x),
    \qquad
    (\mathrm{d}f)_x(v)=2B_f(x,v).
    \label{eq:polar-form}
\end{equation}

\begin{lemma}
\label{lem:quadratic-nondegeneracy}
Let $f\in S_2$. Then $f\in U_2(\mathcal A)$ if and only if
the restriction $B_f|_X$ is nondegenerate for every
$X\in L(\mathcal A)$. In particular, if $f\in U_2(\mathcal A)$,
then $B_f$ is nondegenerate on $V$.
\end{lemma}

\begin{proof}
Fix $X\in L(\mathcal A)$. For $x,v\in X$,
\[
    \mathrm{d}(f|_X)_x(v)
    =
    (\mathrm{d}f)_x(v)
    =
    2B_f(x,v).
\]
Hence the zero locus of $\operatorname{Jac}_X(f|_X)$ in $X$ is
\[
    \{x\in X:B_f(x,v)=0\text{ for all }v\in X\}.
\]
Therefore $f|_X$ is nondegenerate on $X$ if and only if
$B_f|_X$ is nondegenerate. The result follows by considering all
$X\in L(\mathcal A)$.
\end{proof}

Since $T\in L(\mathcal A)$,
Lemma~\ref{lem:quadratic-nondegeneracy} implies that
$B_\eta|_T$ is nondegenerate. Thus
$V=T\oplus W$ for $W:=T^{\perp_{B_\eta}}$.
By the construction of essentialization in 
Section~\ref{subsec:arrangements-OS}, we identify
$\mathcal A^{\mathrm{ess}}$ with
$\mathcal A_W:=\{H\cap W:H\in\mathcal A\}$.

\begin{lemma}\label{lem:eta-essentialization}
Let $\eta_W:=\eta|_W$. Then
$\eta_W\in U_2(\mathcal A_W)$, and
\[
    \mathrm{ST}(\mathcal A,\eta)
    \cong \mathrm{ST}(\mathcal A_W,\eta_W),
    \qquad
    \mathrm{OS}(\mathcal A)
    \cong \mathrm{OS}(\mathcal A_W).
\]
\end{lemma}

\begin{proof}
For every $Y\in L(\mathcal A_W)$, the corresponding
flat of $\mathcal A$ is $T\oplus Y$.
Since $B_\eta|_{T\oplus Y}$ and $B_\eta|_T$
are nondegenerate and $T$ is $B_\eta$-orthogonal
to $W$, the restriction $B_\eta|_Y$ is also
nondegenerate. Hence
$\eta_W\in U_2(\mathcal A_W)$ by
Lemma~\ref{lem:quadratic-nondegeneracy}.

Set $\eta_T:=\eta|_T$. The orthogonal decomposition
gives $\eta=\eta_W+\eta_T$ under the identification
$S\cong S_W\otimes_{\mathbb C}S_T$.
Applying the derivations in the decomposition of
Lemma~\ref{lem:product-log-derivations} to $\eta$, we obtain
\[
    \mathfrak a(\mathcal A,\eta)
    =
    \mathfrak a(\mathcal A_W,\eta_W)S
    +
    \operatorname{Jac}_T(\eta_T)S.
\]
Since $\eta_T$ is a nondegenerate quadratic form,
its partial derivatives span $T^*$, so
$S_T/\operatorname{Jac}_T(\eta_T)\cong\mathbb C$.
Consequently,
\[
    \mathrm{ST}(\mathcal A,\eta)
    \cong
    \mathrm{ST}(\mathcal A_W,\eta_W)
    \otimes_{\mathbb C}
    S_T/\operatorname{Jac}_T(\eta_T)
    \cong
    \mathrm{ST}(\mathcal A_W,\eta_W).
\]

Finally, essentialization preserves the dependent subsets of the arrangement,
so the definition of the Orlik--Solomon algebra gives
$\mathrm{OS}(\mathcal A)\cong\mathrm{OS}(\mathcal A_W)$.
\end{proof}

If $r(\mathcal A)=0$, then $\mathcal A=\varnothing$, and
$\mathrm{ST}(\mathcal A,\eta)\cong
\mathrm{OS}(\mathcal A)\cong\mathbb C$.
We may therefore assume that $r(\mathcal A)>0$.
By Lemma~\ref{lem:eta-essentialization}, we may replace
$(\mathcal A,\eta)$ by $(\mathcal A_W,\eta_W)$
without changing either side of the desired inequality.
Thus, for the remainder of the proof of
Theorem~\ref{thm:main2}, we assume that
$\mathcal A$ is essential, so that $T=0$ and
$r=\ell\geq1$.

We now define
\[
    Y_\eta
    :=
    \{x\in U(\mathcal A):\eta(x)=1\}.
\]
Recall that an irreducible complex algebraic variety is \emph{very affine} if it is isomorphic to a closed subvariety of an
algebraic torus $(\mathbb C^*)^N$ for some $N$. We write $\chi(Y)$ for
the topological Euler characteristic of a complex algebraic variety $Y$.

\begin{lemma}
\label{lem:quadric-section-very-affine}
If $r\geq2$, then $Y_\eta$ is a smooth very affine variety of
dimension $r-1$.
\end{lemma}

\begin{proof}
Since $\eta$ is nondegenerate and $r\geq2$,
the hypersurface $\{\eta=1\}$ is smooth and
irreducible of dimension $r-1$.
Hence its nonempty open subvariety $Y_\eta$
has the same properties.
To show that $Y_\eta$ is very affine, we use a standard
torus embedding. Let $N=|\mathcal A|$ and consider
\[
    \Phi:V\longrightarrow\mathbb C^N,
    \qquad
    x\longmapsto(\alpha_H(x))_{H\in\mathcal A}.
\]
Since $\mathcal A$ is essential,
$\ker\Phi=\bigcap_{H\in\mathcal A}H=0$, so $\Phi$ identifies $V$
with a closed linear subspace of $\mathbb C^N$. Moreover, $\Phi$ identifies $U(\mathcal A)$ with
$\Phi(V)\cap(\mathbb C^*)^N$, which is a closed subvariety of
$(\mathbb C^*)^N$. Since \(Y_{\eta}\)
is closed in $U(\mathcal A)$, it is also a closed subvariety of
$(\mathbb C^*)^N$. Therefore $Y_\eta$ is very affine.
\end{proof}

We next compute the Euler characteristic of $Y_\eta$.

\begin{proposition}
\label{prop:quadric-euler-OS}
We have $(-1)^{r-1}\chi(Y_\eta)
=\dim_{\mathbb C}\mathrm{OS}(\mathcal A)$.
\end{proposition}

\begin{proof}
For $x\in V$, set
\[
    X_x:=\bigcap_{\substack{H\in\mathcal A\\ x\in H}} H.
\]
For every $X\in L(\mathcal A)$, we have the stratification into locally closed subvarieties
\[
    \{x\in X:\eta(x)=1\}
    =
    \bigsqcup_{X\leq Y}
    \{x\in V:\eta(x)=1,\ X_x=Y\},
\]
where the order on $L(\mathcal A)$ is reverse inclusion. By the
set-theoretic additivity of the topological Euler characteristic for
complex algebraic varieties \cite[Section~4.5]{Fulton1993toric},
\[
    \chi\bigl(\{x\in X:\eta(x)=1\}\bigr)
    =
    \sum_{X\leq Y}
    \chi\bigl(\{x\in V:\eta(x)=1,\ X_x=Y\}\bigr).
\]
Applying Möbius inversion and taking $X=V$, we obtain
\[
    \chi(Y_\eta)
    =
    \sum_{X\in L(\mathcal A)}
    \mu(X)\,
    \chi\bigl(\{x\in X:\eta(x)=1\}\bigr),
\]
since $X_x=V$ if and only if $x\in U(\mathcal A)$.

For $k>0$, after a linear change of coordinates, the level set of a
nondegenerate quadratic form on $\mathbb C^k$ is diffeomorphic to
$TS^{k-1}$ \cite[Section~2]{Szoke1999}, hence its Euler characteristic is $1-(-1)^k$.
For $k=0$, the same formula holds trivially. Therefore
\[
\begin{aligned}
    \chi(Y_\eta)
    &=
    \sum_{X\in L(\mathcal A)}
    \mu(X)\bigl(1-(-1)^{\dim X}\bigr) \\
    &=
    \chi_{\mathcal A}(1)-\chi_{\mathcal A}(-1)
    =
    -\chi_{\mathcal A}(-1),
\end{aligned}
\]
where $\chi_{\mathcal A}(1)=0$ since $\mathcal A$ has positive rank.
On the other hand, by \eqref{eq:OS-Hilbert-Poincare},
\[
    \dim_{\mathbb C}\mathrm{OS}(\mathcal A)
    =
    \pi_{\mathcal A}(1)
    =
    (-1)^r\chi_{\mathcal A}(-1),
\]
where the last equality follows from \eqref{eq-Poin} and
$r=\dim V$.
Combining the two identities gives
$(-1)^{r-1}\chi(Y_\eta)
=\dim_{\mathbb C}\mathrm{OS}(\mathcal A)$.
\end{proof}

To turn this Euler characteristic into a point count, we use the
following theorem of Huh. We first recall some relevant definitions. Let $Y$ be a smooth very affine variety of
dimension $m$, and fix a closed embedding
\[
    f=(f_1,\ldots,f_N)\colon Y\hookrightarrow(\mathbb C^*)^N.
\]
For $\mathbf u=(u_1,\ldots,u_N)\in\mathbb Z^N$, the associated
\emph{master function} is defined by
$\varphi_{\mathbf u}:=\prod_{i=1}^N f_i^{u_i}$.
Since $\varphi_{\mathbf u}$ does not vanish on $Y$, its critical
points are precisely the zeros of the logarithmic one-form
\[
    \mathrm{d}\log\varphi_{\mathbf u}
    =
    \sum_{i=1}^N u_i\,\frac{\mathrm{d}f_i}{f_i},
\]
as
$\mathrm{d}\varphi_{\mathbf u}
=\varphi_{\mathbf u}\,\mathrm{d}\log\varphi_{\mathbf u}$.

\begin{theorem}[{\cite[Theorem~1]{huh2013maximum}}]
\label{thm:Huh-ML-degree}
In the setting above, there exists a 
nonzero polynomial $F\in\mathbb C[u_1,\dots,u_N]$ such that, for every $\mathbf u\in \mathbb Z^N$ with $F(\mathbf u)\neq 0$, $\varphi_{\mathbf u}$ has exactly
$(-1)^m\chi(Y)$ critical points on $Y$, all of which are nondegenerate.
\end{theorem}

Applying Theorem~\ref{thm:Huh-ML-degree} to $Y_\eta$ gives the
following.

\begin{proposition}
\label{prop:log-form-zeros}
There exist weights $\lambda_H\in\mathbb Q$, $H\in\mathcal A$, satisfying
$\sum_{H\in\mathcal A}\lambda_H=2$ such that, for the logarithmic
one-form
\[
    \omega
    :=
    \sum_{H\in\mathcal A}
    \lambda_H\frac{\mathrm{d}\alpha_H}{\alpha_H}
\]
defined on $U(\mathcal A)$, the restriction $\omega|_{Y_\eta}$ has exactly
$\dim_{\mathbb C}\mathrm{OS}(\mathcal A)$ distinct zeros.
\end{proposition}

\begin{proof}
Assume first that $r\geq2$. By
Lemma~\ref{lem:quadric-section-very-affine}, $Y_\eta$ is a smooth
very affine variety of dimension $r-1$.
Applying Theorem~\ref{thm:Huh-ML-degree} to the embedding from
Lemma~\ref{lem:quadric-section-very-affine}, we may choose
$\mathbf u=(u_H)_{H\in\mathcal A}\in\mathbb Z^N$, where
$N:=|\mathcal A|$, such that
$c:=\sum_{H\in\mathcal A}u_H\neq0$ and the logarithmic one-form
\[
    \widetilde{\omega}
    :=
    \sum_{H\in\mathcal A}
    u_H\frac{\mathrm{d}\alpha_H}{\alpha_H}
\]
has exactly
\((-1)^{r-1}\chi(Y_\eta)=\dim_{\mathbb C}\mathrm{OS}(\mathcal A)\)
distinct zeros on $Y_\eta$, where the equality follows from
Proposition~\ref{prop:quadric-euler-OS}.

Set $\lambda_H:=2u_H/c\in\mathbb Q$. Then
$\sum_H\lambda_H=2$, and the corresponding one-form satisfies
$\omega=(2/c)\widetilde{\omega}$. Hence
$\omega|_{Y_\eta}$ and $\widetilde{\omega}|_{Y_\eta}$ have the same
zero set, proving the proposition for $r\geq2$.

If $r=1$, then $Y_\eta$ consists of two
points and $\dim_{\mathbb C}\mathrm{OS}(\mathcal A)=2$. Taking $\lambda_H=2$ gives $\omega|_{Y_\eta}=0$, since $Y_\eta$ is
$0$-dimensional. The proposition follows.
\end{proof}

For the remainder of the proof of Theorem~\ref{thm:main2},
we fix such weights and the corresponding logarithmic
one-form $\omega$.

We now construct an ideal whose zero locus contains the zeros of
$\omega|_{Y_\eta}$. For $\theta\in D^1(\mathcal A)$, logarithmicity gives
$\theta(\alpha_H)/\alpha_H\in S$ for every $H\in\mathcal A$.
Thus evaluating $\omega$ on $\theta$ yields a polynomial
\[
    \omega(\theta)
    :=
    \sum_{H\in\mathcal A}
    \lambda_H\frac{\theta(\alpha_H)}{\alpha_H}
    \in S.
\]
Define
\[
    \mathfrak a_1
    :=
    \bigl(
        \theta(\eta)-\omega(\theta):
        \theta\in D^1(\mathcal A)
    \bigr)
    \subseteq S.
\]

\begin{proposition}
\label{prop:auxiliary-ideal-zeros}
The zero locus of $\mathfrak a_1$ contains at least
$\dim_{\mathbb C}\mathrm{OS}(\mathcal A)$ distinct points.
\end{proposition}

\begin{proof}
Let $y\in Y_\eta$ be a zero of $\omega|_{Y_\eta}$. Since
$Y_\eta$ is an open subvariety of $\{\eta=1\}$, we have
$T_yY_\eta=\ker(\mathrm{d}\eta)_y$, and hence
$\omega_y=c(\mathrm{d}\eta)_y$ for some $c\in\mathbb C$.
To determine $c$, let 
$E=\sum_{i=1}^r x_i\partial_{x_i}$ 
be the
Euler vector field. Since each $\alpha_H$ is linear,
\[
    \omega_y(E_y)=\sum_H\lambda_H=2,
    \qquad
    (\mathrm d\eta)_y(E_y)=2\eta(y)=2.
\]
Thus $c=1$, and hence
$\omega_y=(\mathrm d\eta)_y$.
For every $\theta\in D^1(\mathcal A)$, viewing $\theta$ as a
polynomial vector field, we have
\[
    \theta(\eta)(y)-\omega(\theta)(y)
    =
    \bigl((\mathrm{d}\eta)_y-\omega_y\bigr)(\theta(y))
    =
    0.
\]
Thus $y$ is a common zero of $\mathfrak a_1$, and the result follows
from Proposition~\ref{prop:log-form-zeros}.
\end{proof}

This point count gives the following estimate by the Chinese remainder theorem.

\begin{proposition}
\label{prop:OS-lower-bound}
We have $\dim_{\mathbb C}S/\mathfrak a_1
\geq\dim_{\mathbb C}\mathrm{OS}(\mathcal A)$.
\end{proposition}

\begin{proof}
By Proposition~\ref{prop:auxiliary-ideal-zeros}, the zero
locus of $\mathfrak a_1$ contains at least
$\dim_{\mathbb C}\mathrm{OS}(\mathcal A)$ distinct points. Let $k=\dim_{\mathbb C}\mathrm{OS}(\mathcal A)$ and choose
$k$ such points $x_1,\ldots,x_k$. Let
$\mathfrak m_{x_i}\subset S$ be the corresponding maximal ideals, so \(\mathfrak a_1\subseteq\bigcap_{i=1}^k\mathfrak m_{x_i}\).
Since the ideals $\mathfrak m_{x_i}$ are pairwise coprime, the Chinese
remainder theorem gives a surjection
\[
    S/\mathfrak a_1
    \twoheadrightarrow
    \bigoplus_{i=1}^k S/\mathfrak m_{x_i}
    \cong
    \mathbb C^k.
\]
Hence
$\dim_{\mathbb C}S/\mathfrak a_1
\geq k=\dim_{\mathbb C}\mathrm{OS}(\mathcal A)$.
\end{proof}

We now show that
\[
    \dim_{\mathbb C}S/\mathfrak a_1
    \leq
    \dim_{\mathbb C}\mathrm{ST}(\mathcal A,\eta)
\]
by comparing the Solomon--Terao ideal with the degree form ideal of
$\mathfrak a_1$.
For a nonzero polynomial $f\in S$, write it as a sum of homogeneous
components,
\[
    f=f_0+f_1+\cdots+f_d,
    \qquad
    f_i\in S_i,\quad f_d\neq0.
\]
Its \emph{degree form} is $\mathrm{DF}(f):=f_d$. For an ideal
$I\subseteq S$, define its \emph{degree form ideal} by
\[
    \mathrm{DF}(I)
    :=
    \bigl(
        \mathrm{DF}(f):0\neq f\in I
    \bigr).
\]

\begin{proposition}
\label{prop:degree-comparison}
We have $\dim_{\mathbb C}S/\mathfrak a_1
\leq\dim_{\mathbb C}\mathrm{ST}(\mathcal A,\eta)$.
\end{proposition}

\begin{proof}
Since $D^1(\mathcal A)$ is graded, $\mathfrak a(\mathcal A,\eta)$ is
generated by $\theta(\eta)$ with $\theta\in D^1(\mathcal A)$ homogeneous.
If $\theta$ is homogeneous of degree $e$ and $\theta(\eta)\neq0$,
then $\theta(\eta)$ is homogeneous of degree $e+1$, while
$\omega(\theta)$ is either zero or homogeneous of degree $e-1$.
Hence
\[
    \mathrm{DF}
    \bigl(\theta(\eta)-\omega(\theta)\bigr)
    =
    \theta(\eta).
\]
Since $\theta(\eta)-\omega(\theta)\in\mathfrak a_1$, it follows that
\(\mathfrak a(\mathcal A,\eta)\subseteq\mathrm{DF}(\mathfrak a_1)\).

By Proposition~\ref{prop:ST-finite-homology},
$S/\mathfrak a(\mathcal A,\eta)=\mathrm{ST}(\mathcal A,\eta)$ is
finite-dimensional, so the inclusion above implies that
$S/\mathrm{DF}(\mathfrak a_1)$ is finite-dimensional.
By \cite[Theorem~4.3.22]{KreuzerRobbiano2005}, applied to the
standard grading,
\[
    \dim_{\mathbb C}S/\mathfrak a_1
    =
    \dim_{\mathbb C}S/\mathrm{DF}(\mathfrak a_1).
\]
Therefore
\[
    \dim_{\mathbb C}S/\mathfrak a_1
    \leq
    \dim_{\mathbb C}S/\mathfrak a(\mathcal A,\eta)
    =
    \dim_{\mathbb C}\mathrm{ST}(\mathcal A,\eta).\qedhere
\]
\end{proof}

Finally, Proposition~\ref{prop:OS-lower-bound} and
Proposition~\ref{prop:degree-comparison} immediately imply
Theorem~\ref{thm:main2}.

\begin{remark}\label{rem:essential-higher-degree}
The preceding arguments extend to higher degrees for essential arrangements. More precisely, if $\mathcal A$ is essential in $\mathbb C^\ell$, then for $d\geq2$ and $\eta\in U_d(\mathcal A)$,
\[
    \dim_{\mathbb C}\mathrm{ST}(\mathcal A,\eta)
    \geq (-1)^\ell\chi_{\mathcal A}(1-d).
\]
Indeed, one can show that $\chi(Y_\eta)=-\chi_{\mathcal A}(1-d)$. With $\sum_H\lambda_H=d$, the same critical-point argument and degree-form comparison yield the desired inequality. 
\end{remark}

\begin{remark}\label{rem:ST-OS-strict}
The inequality in Theorem~\ref{thm:main2} can be strict. For example,
let $G=K_{2,2,2}$ be the complete tripartite graph on six vertices,
and let $\mathcal A_G$ be its graphic arrangement. For
$\eta=x_1^2+\cdots+x_6^2$, a SageMath computation gives
\[
    \dim_{\mathbb C}\mathrm{ST}(\mathcal A_G,\eta)
    =
    427
    >
    426
    =
    \dim_{\mathbb C}\mathrm{OS}(\mathcal A_G).
\]
\end{remark}

\begin{remark}\label{rem:ST-OS-equality-nontame}
By Corollary~\ref{coro-ST=OS}, if $\mathcal A$ is tame, then
\(\dim_{\mathbb C}\mathrm{ST}(\mathcal A,\eta)=
\dim_{\mathbb C}\mathrm{OS}(\mathcal A)\)
for every $\eta\in U_2(\mathcal A)$. A naive question is whether the
converse holds, that is, whether this equality implies tameness.

However, the converse is false. Let $\mathcal A$ be the arrangement
in $V=\mathbb C^5$ with defining polynomial
\[
    Q(\mathcal A)
    =
    x_1x_2x_3x_4x_5
    \prod_{\substack{
        (\varepsilon_2,\ldots,\varepsilon_5)\in\{\pm1\}^4\\
        (\varepsilon_2,\ldots,\varepsilon_5)\neq(1,1,1,1)
    }}
    (x_1+\varepsilon_2x_2+\cdots+\varepsilon_5x_5).
\]
This is a deletion from the Edelman--Reiner arrangement. For every
$\eta\in U_2(\mathcal A)$, in particular for
$\eta=x_1^2+\cdots+x_5^2$, one has
\[
    \dim_{\mathbb C}\mathrm{ST}(\mathcal A,\eta)
    =
    \dim_{\mathbb C}\mathrm{OS}(\mathcal A)
    =
    2222.
\]
However, $\operatorname{pd}_S D^4(\mathcal A)=3>5-4$, so
$\mathcal A$ is not tame. We leave the verification of these
claims to the reader.
\end{remark}

\section{Equivalence between tameness and ST-finiteness}

By Proposition~\ref{prop-vanish}, tameness implies
ST-finiteness. In this section, we prove the converse,
completing the proof of Theorem~\ref{thm:main3}.
The key step is a chain isomorphism involving a Koszul
complex, which yields a decomposition of Tor groups.
This decomposition allows us to recover the
projective-dimension bounds defining tameness.

Fix $d\ge2$ and $\eta\in U_d(\mathcal A)$, and set
\[
    z_i:=\frac{\partial\eta}{\partial x_i}
    \quad (1\le i\le\ell),\qquad
    J_\eta:=(z_1,\ldots,z_\ell),\qquad
    R_\eta:=S/J_\eta.
\]
Let $S^\ell=\bigoplus_{i=1}^\ell Se_i$, and let
$K_\bullet=K_\bullet(z_1,\ldots,z_\ell;S)$ denote the
\emph{Koszul complex} associated with $z_1,\ldots,z_\ell$
\cite[Tag 0621]{stacks-project}.
Its terms are
\[
    K_q=\bigwedge\nolimits_S^q S^\ell
    \qquad (0\le q\le\ell),
\]
with $K_q=0$ otherwise. The differential
$\dd_K:K_q\to K_{q-1}$ is given by
\[
\dd_K(e_{i_1}\wedge\cdots\wedge e_{i_q})
=\sum_{r=1}^q(-1)^{r-1}z_{i_r}
e_{i_1}\wedge\cdots\wedge
\widehat{e_{i_r}}\wedge\cdots\wedge e_{i_q},
\]
where the hat denotes omission. In particular,
$K_0=S$ and $\dd_K^2=0$.

\begin{lemma}\label{resolutionfree}
The following statements hold.
\begin{enumerate}[label=(\roman*)]
    \item\label{lem:resolutionfree:regular}
    The sequence $z_1,\ldots,z_\ell$ is a regular sequence of $S$;

    \item\label{lem:resolutionfree:length}
    The ideal $J_\eta$ is contained in
    $\mathfrak m:=(x_1,\ldots,x_\ell)$, and
    $R_\eta$ is a nonzero graded $S$-module
    of finite length;

    \item\label{lem:resolutionfree:koszul}
    The Koszul complex $K_\bullet$ is a free resolution
    of $R_\eta$ of length $\ell$.
\end{enumerate}
\end{lemma}

\begin{proof}\leavevmode\par
\begin{enumerate}[label=(\roman*)]
    \item This follows from
    \cite[Lemma 2.9]{abe2019solomon}.

    \item Since $d\ge2$, the ideal $J_\eta$ is homogeneous
    and contained in $\mathfrak m$, so $R_\eta$ is
    nonzero and graded. By~\ref{lem:resolutionfree:regular}, $J_\eta$ is generated
    by a regular sequence of length $\ell=\dim S$.
    Hence $R_\eta$ has finite length.

    \item This follows from
    \ref{lem:resolutionfree:regular} and
    \cite[Corollary 19.3]{Eisenbud1995}.\qedhere
\end{enumerate}
\end{proof}

For each $1\le i\le\ell$ and $1\le p\le\ell$,
define an $S$-linear map
\[
    \delta_i:D^p(\mathcal A)\longrightarrow
    \operatorname{Der}^{p-1}S,\quad
    (\delta_i\theta)(f_2,\ldots,f_p)
    :=\theta(x_i,f_2,\ldots,f_p).
\]
For $p=0$, set $\delta_i=0$.
The following lemma shows that these maps
take values in $D^{p-1}(\mathcal A)$
and records their basic properties.

\begin{lemma}\label{decomppartial}
The maps $\delta_i$ satisfy the following properties.
\begin{enumerate}[label=(\roman*)]
    \item\label{lem:delta:preserves}
    We have $\delta_i(D^p(\mathcal A))
    \subseteq D^{p-1}(\mathcal A)$;

    \item\label{lem:delta:anticommute}
    For every $i$, $\partial\delta_i+\delta_i\partial=0$;

    \item\label{lem:delta:decomposition}
    The differential $\partial$ decomposes as
    $\partial=\sum_{i=1}^{\ell}z_i\delta_i$.
\end{enumerate}
\end{lemma}

\begin{proof}\leavevmode\par
\begin{enumerate}[label=(\roman*)]
    \item Let $\theta\in D^p(\mathcal A)$ and
    $H\in\mathcal A$.
    For arbitrary $f_3,\ldots,f_p\in S$
    and $p\ge2$, we have
    \[
    \begin{aligned}
        (\delta_i\theta)(\alpha_H,f_3,\ldots,f_p)
        &=\theta(x_i,\alpha_H,f_3,\ldots,f_p)\\
        &=-\theta(\alpha_H,x_i,f_3,\ldots,f_p)
        \in\alpha_HS.
    \end{aligned}
    \]
    Thus $\delta_i\theta\in D^{p-1}(\mathcal A)$.
    The cases $p=0,1$ follow directly
    from the definitions.

    \item For $p\ge2$, compute
    \[
    \begin{aligned}
        (\partial\delta_i\theta)(f_3,\ldots,f_p)
        &=\theta(x_i,\eta,f_3,\ldots,f_p),\\
        (\delta_i\partial\theta)(f_3,\ldots,f_p)
        &=\theta(\eta,x_i,f_3,\ldots,f_p).
    \end{aligned}
    \]
    Hence $\partial\delta_i+\delta_i\partial=0$
    for $p\ge2$.
    The cases $p=0,1$ follow directly
    from the definitions.

    \item For $p=0$, the identity holds by definition. For $p\ge1$, the chain rule gives
    \[
        (\partial\theta)(f_2,\ldots,f_p)
        =\sum_{i=1}^\ell z_i(\delta_i\theta)(f_2,\ldots,f_p).\qedhere
    \]
\end{enumerate}
\end{proof}

By the definition of the Koszul differential,
$\dd_K(e_i)=z_i$. Hence, for $a\in K_q$,
\begin{equation}\label{eq:koszul-wedge}
    \dd_K(e_i\wedge a)=z_i a-e_i\wedge\dd_K a.
\end{equation}
Consider the total complex
\[
    T_n:=\bigoplus_{p+q=n}K_q\otimes_S D^p(\mathcal A).
\]
For $a\in K_q$ and $c\in D^p(\mathcal A)$, define
\[
    \dd_0(a\otimes c):=\dd_K a\otimes c, \qquad
    \dd_1(a\otimes c):=(-1)^q a\otimes\partial c.
\]
Then $\dd:=\dd_0+\dd_1$ satisfies $\dd^2=0$ and
defines the total differential on
$T_\bullet=\operatorname{Tot}
(K_\bullet\otimes_S D^*(\mathcal A))$.

\begin{proposition}
\label{prop:chainisom}
    There is an $S$-linear chain isomorphism
    \[
        (T_\bullet,\dd_0+\dd_1)\cong(T_\bullet,\dd_0).
    \]
\end{proposition}

\begin{proof}   
    For $1\le i\le\ell$, define an $S$-linear
    endomorphism $E_i:T_\bullet\to T_\bullet$ by
    \[
        E_i(a\otimes c)
        :=(-1)^q(e_i\wedge a)\otimes\delta_i c,
        \qquad a\in K_q,\ c\in D^p(\mathcal A).
    \]
    Set $E:=\sum_{i=1}^\ell E_i$. Since $E$
    increases the Koszul degree by one and
    $K_q=0$ for $q>\ell$, we have $E^{\ell+1}=0$.
    Thus
    \[
        \Phi:=\exp(E)=\sum_{r=0}^\ell\frac{E^r}{r!}
    \]
    is an $S$-linear automorphism of $T_\bullet$
    with inverse $\exp(-E)$.
    
    Now we compute $[\dd_0,E]$ and $[\dd_1,E]$,
    where $[A,B]:=AB-BA$.    
    For $a\in K_q$ and $c\in D^p(\mathcal A)$,
    using~\eqref{eq:koszul-wedge}, we obtain
    \[
        [\dd_0,E_i](a\otimes c)
        =(-1)^q\bigl(\dd_K(e_i\wedge a)
        +e_i\wedge\dd_K a\bigr)\otimes\delta_i c
        =(-1)^q z_i a\otimes\delta_i c.
    \]
    By Lemma~\ref{decomppartial}\ref{lem:delta:decomposition},
    we obtain $[\dd_0,E]=\dd_1$.
    Similarly,
    \[
        [\dd_1,E_i](a\otimes c)
        =-e_i\wedge a\otimes
        (\partial\delta_i+\delta_i\partial)c=0
    \]
    by Lemma~\ref{decomppartial}\ref{lem:delta:anticommute}.
    Hence $[\dd_1,E]=0$.
    
    Since $E$ is nilpotent,
    $F(t):=\exp(-tE)\dd_0\exp(tE)$ has only finitely many terms upon expansion. Using the commutator identities above,
    we obtain
    \[
    F'(t)
    =\exp(-tE)[\dd_0,E]\exp(tE)
    =\exp(-tE)\dd_1\exp(tE)=\dd_1.
    \]
    Since $F(0)=\dd_0$ and
    $F(1)=\Phi^{-1}\dd_0\Phi$, it follows that
    $\dd_0\Phi=\Phi(\dd_0+\dd_1)$.
    Thus $\Phi$ is the desired chain isomorphism.
\end{proof}

\begin{proposition}\label{prop:decomptor}
Assume that $\mathcal A$ is ST-finite with respect to
$\eta$. Then, for every $n\ge0$, there is an isomorphism
of $S$-modules
\[
    \mathrm{Tor}_n^S(R_\eta,\mathrm{ST}(\mathcal A,\eta))
    \cong
    \bigoplus_{p+q=n}
    \mathrm{Tor}_q^S(R_\eta,D^p(\mathcal A)).
\]
\end{proposition}

\begin{proof}
Let $L_\bullet$ be the chain complex concentrated
in degree zero with
$L_0=\mathrm{ST}(\mathcal A,\eta)$.
By ST-finiteness, the natural chain map
$D^*(\mathcal A)\to L_\bullet$ is a quasi-isomorphism.

Since $K_\bullet$ is bounded and degreewise free,
it is K-flat \cite[Tag~064K]{stacks-project}.
Hence tensoring $D^*(\mathcal A)\to L_\bullet$
with $K_\bullet$ and taking total complexes
gives a quasi-isomorphism
\cite[Tag~06Y0]{stacks-project}.
By Lemma~\ref{resolutionfree}\ref{lem:resolutionfree:koszul},
we obtain
\[
    H_n(T_\bullet,\dd_0+\dd_1)
    \cong H_n(K_\bullet\otimes_S
    \mathrm{ST}(\mathcal A,\eta))
    \cong \mathrm{Tor}_n^S
    (R_\eta,\mathrm{ST}(\mathcal A,\eta)).
\]

By Proposition~\ref{prop:chainisom}, we have
$H_n(T_\bullet,\dd_0+\dd_1)
\cong H_n(T_\bullet,\dd_0)$.
Since $\dd_0$ preserves the index $p$,
\[
    H_n(T_\bullet,\dd_0)
    \cong\bigoplus_{p+q=n}
    H_q(K_\bullet\otimes_S D^p(\mathcal A))
    \cong\bigoplus_{p+q=n}
    \mathrm{Tor}_q^S(R_\eta,D^p(\mathcal A)).
\]
Combining these isomorphisms proves the assertion.
\end{proof}

The following lemma gives a description of the projective dimension.

\begin{lemma}\label{tordetector}
Let $N$ be a nonzero graded $S$-module of finite
length and $M$ a nonzero finitely generated graded
$S$-module. Then
\[
    \sup\{q\ge0\mid\operatorname{Tor}_q^S(N,M)\neq0\}
    =\operatorname{pd}_S M.
\]
\end{lemma}

\begin{proof}
    By Hilbert's syzygy theorem, $M$ has finite
    projective dimension over $S$.
    Set $t:=\operatorname{pd}_S M$ and choose a minimal
    graded free resolution $F_\bullet$ of $M$
    of length $t$, with $F_t\neq0$.
    For $t>0$, it has the form
    \[
        0\longrightarrow F_t\xrightarrow{f_t}
        F_{t-1}\longrightarrow\cdots
        \longrightarrow F_0\longrightarrow M
        \longrightarrow0.
    \]
    By minimality, $f_j(F_j)\subseteq\mathfrak m F_{j-1}$
    for all $j\ge1$, where
    $\mathfrak m=(x_1,\ldots,x_\ell)$.

    Since $N$ is nonzero and of finite length,
    it is Artinian. Thus the descending chain
    \[
        N\supseteq\mathfrak mN
        \supseteq\mathfrak m^2N\supseteq\cdots
    \]
    stabilizes. By graded Nakayama's lemma,
    the stable term must be zero.
    Hence there exists $r\ge0$ such that
    $\mathfrak m^rN\neq0$ and
    $\mathfrak m^{r+1}N=0$.
    Choose $0\neq u\in\mathfrak m^rN$.
    Then $\mathfrak m u=0$.

    If $t=0$, then $M$ is a nonzero free
    $S$-module. Hence
    $\operatorname{Tor}_0^S(N,M)=N\otimes_S M\neq0$
    and $\operatorname{Tor}_i^S(N,M)=0$
    for all $i>0$.
    If $t>0$, choose a basis element $e\in F_t$.
    By minimality, $f_t(e)$ is an $S$-linear
    combination of basis elements of $F_{t-1}$
    with coefficients in $\mathfrak m$.
    Since $\mathfrak m u=0$, we have
    \[
        (f_t\otimes 1_N)(e\otimes u)=0.
    \]
    The element $e\otimes u$ is nonzero in
    $F_t\otimes_S N$. Since $F_{t+1}=0$,
    it represents a nonzero homology class.
    Thus
    \[
        0\neq[e\otimes u]\in
        H_t(F_\bullet\otimes_S N)
        =\operatorname{Tor}_t^S(M,N)
        \cong\operatorname{Tor}_t^S(N,M).
    \]
    Since $F_\bullet$ has length $t$,
    $\operatorname{Tor}_i^S(N,M)=0$ for all $i>t$.
\end{proof}

We now complete the proof of Theorem~\ref{thm:main3}.

\begin{proof}[Proof of Theorem~\ref{thm:main3}]
By Proposition~\ref{prop-vanish}, it suffices to
show that ST-finiteness implies tameness.
Assume that $\mathcal A$ is ST-finite with respect
to $\eta$, and fix $1\le p\le\ell$.
Recall that $D^p(\mathcal A)$ is a nonzero
finitely generated graded $S$-module.

For every $q>\ell-p$, we have $p+q>\ell$.
Since $K_\bullet$ is a free resolution of $R_\eta$
of length $\ell$, we have
$\mathrm{Tor}_{p+q}^S(R_\eta,\mathrm{ST}(\mathcal A,\eta))=0$.
By Proposition~\ref{prop:decomptor},
\[
    \bigoplus_{a+b=p+q}
    \mathrm{Tor}_b^S(R_\eta,D^a(\mathcal A))=0.
\]
Hence $\mathrm{Tor}_q^S(R_\eta,D^p(\mathcal A))=0$
for all $q>\ell-p$.

By Lemma~\ref{resolutionfree}\ref{lem:resolutionfree:length},
$R_\eta$ is a nonzero finite-length graded
$S$-module. 
Combining this with the fact that
$D^p(\mathcal A)$ is a nonzero finitely generated
graded $S$-module, we obtain from
Lemma~\ref{tordetector} that
\[
    \operatorname{pd}_S D^p(\mathcal A)\le\ell-p.
\]
Since this holds for every $1\le p\le\ell$,
the arrangement $\mathcal A$ is tame.
\end{proof}

\bibliographystyle{amsplain}
\bibliography{hyperplane}

\end{document}